\documentclass[12pt,twoside]{article}
\usepackage{amsmath,amssymb}
\usepackage{graphicx}
\usepackage[footnotesize]{subfigure}
\usepackage[footnotesize]{caption2}
\usepackage{authblk}
\usepackage{amsthm}

\allowdisplaybreaks[4]

\newtheorem{lemma}{{\bf Lemma}}[section]

\newtheorem{theorem}{{\bf Theorem}}[section]

\newtheorem{corollary}{{\bf Corollary}}[section]

\makeatletter
\evensidemargin
\oddsidemargin
\makeatother

\title{
\Large\bf  Effective reduction of  a three-dimensional circadian\\ oscillator model}
\author{{\sc Shuang Chen$^{a,b}$}, {\sc Jinqiao Duan$^{c}$}, {\sc Ji Li$^{a}$}
\\
{\small $^{a}$ School of Mathematics and Statistics, Huazhong University of Sciences and Technology}\\
{\small Wuhan, Hubei 430074, P. R. China}\\
{\small $^{b}$ Center for Mathematical Sciences, Huazhong University of Sciences and Technology}\\
{\small Wuhan, Hubei 430074, P. R. China}\\
{\small $^{c}$ Department of Applied Mathematics, Illinois Institute of Technology,}\\
{\small Chicago, IL 60616, USA}
}

\date{}

\begin{document}
\maketitle
\begin{abstract}

We investigate the dynamics of a three-dimensional system modeling a molecular mechanism for
the circadian rhythm in Drosophila.
We first prove the existence of a compact attractor in the region with biological meaning.
Under the assumption that the dimerization reactions are fast,
in this attractor we reduce the three-dimensional system
to a simpler two-dimensional system on the persistent normally hyperbolic slow manifold.

\vskip 0.2cm
{\bf AMS(2020) Subject Classification}: 34C45; 34A26; 34C20.
\end{abstract}
\baselineskip 15pt
\parskip 10pt

\thispagestyle{empty}
\setcounter{page}{1}


\section{Introduction}\label{sec-intr}

\setcounter{theorem}{0}

Circadian oscillators  display the rhythms of physiology     with a period of about 24 hours. They have been widely found in biological systems and organisms,  such as   fruit flies, plants,  and invertebrate and vertebrate animals
  \cite{Dunlap-99,Gonze-11}.
Much effort has been devoted  to understanding the mechanisms for these oscillations,
and we refer to several excellent references \cite{Forger-17,Keener-Sneyd-98, Rubin-Terman}.

Based on dimerization and proteolysis of PER and TIM proteins in Drosophila,
Tyson et al. \cite{Tyson-etal-99}   set up a simple circadian oscillator model,
which is governed by the following three-dimensional system
\begin{subequations}
\label{3D-model-1}
\begin{align}
\frac{d M}{d t} &= \frac{\nu_m}{1+(P_2/P_c)^{2}}-k_m M,
\label{3D-model-1-1}
\\
\frac{d P_1}{d t} &= \nu_p M-\frac{k_1P_1}{J_p+P_1+rP_2}-k_3P_1-2k_a P_1^{2}+2k_d P_2,
\label{3D-model-1-2}
\\
\frac{d P_2}{d t} &= k_a P_1^{2}-k_d P_2-\frac{k_2P_2}{J_p+P_1+rP_2}-k_3P_2,
\label{3D-model-1-3}
\end{align}
\end{subequations}
where the system states $M$, $P_{1}$ and $P_{2}$ denote the concentration of mRNA, monomer and dimer, respectively. The system parameters are:
$v_{m}$ is the maximum rate of synthesis of mRNA,
$k_{m}$ is the first-order rate degradation of mRNA,
$P_{c}$ is the value of dimer at the half-maximum transcription rate,
$v_{p}$ is the rate for translation of mRNA into the monomer,
$k_{1}$ (resp. $k_{2}$) is the maximum rate for monomer (resp. dimer) phosphorylation,
$k_{3}$ is the first-order rate degradation of the monomer and dimer,
$J_{P}$ is the Michaelis constant for protein kinase DBT,
$k_{a}$ is the rate of dimerization,
$k_{d}$ is the rate of dissociation of the dimer, and finally,
$r$ is the ratio of enzyme-substrate dissociation constants for the monomer and dimer
(see \cite[Table 1]{Tyson-etal-99}).
The mechanism for model (\ref{3D-model-1}) is shown in Figure \ref{fig-CR}.
\begin{figure}
  \centering
  \includegraphics[width=12.4cm]{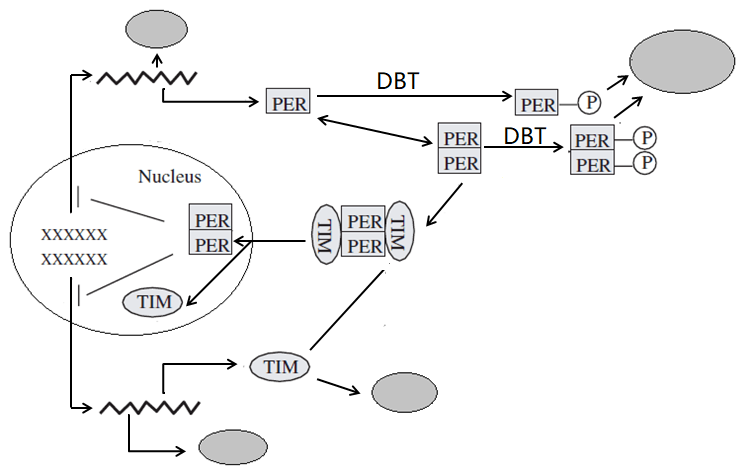}
  \caption{The mechanism for the circadian oscillator model (\ref{3D-model-1}).
  Adapted from \cite{Tyson-etal-99}.}\label{fig-CR}
\end{figure}
Here as in \cite{Tyson-etal-99},
we assume that the ratio $r$ is chosen to be $2$ and monomer is phosphorylated rapidly than dimer, that is, $k_{1}>k_{2}$.
Let $P=P_1+2 P_2$ denote the total protein.
Then system (\ref{3D-model-1}) is converted into the following system
\begin{subequations}
\label{3D-model-2}
\begin{align}
\frac{d M}{d t} &= \frac{4\nu_m P_{c}^{2}}{4P_{c}^{2}+(P-P_1)^{2}}-k_m M,
\label{3D-model-2-1}
\\
\frac{d P}{d t} &= \nu_p M-\frac{(k_1-k_2)P_1+k_2P}{J_p+P}-k_3P,
\label{3D-model-2-2}
\\
\frac{d P_1}{d t} &= \nu_p M-\frac{k_1P_1}{J_p+P}-k_3P_1-2k_a P_1^{2}-k_d P_1+k_dP.
\label{3D-model-2-3}
\end{align}
\end{subequations}
We observe that  the set $\mathcal{A}_{+}:=\left\{(M,P,P_1)\in \mathbb{R}^{3}: M\geq 0, P\geq P_1\geq 0\right\}$
is  the region with biological meaning.

To understand the dynamical behaviours of high-dimensional models arising from  biological, chemical and physical phenomena,
such as the three-dimensional system (\ref{3D-model-2}),
the method of lower dimensional reduction is effective in simplifying these models.
Several methods have been  recently developed to treat the problem of   model reduction,
for example,   the quasi-steady-state approximation \cite{Boieetal-16,Goeke-Walcher-Zerz-15},
computational singular perturbation \cite{Goussis-Najm-06,McMillen-etal-02}, and
intrinsic low-dimensional manifold reduction  \cite{Kaper-02,Maas-Pope}.
Under the assumption that
the rates $k_{a}$ and $k_{d}$ are large compared to other rate parameters,
Tyson et al. \cite{Tyson-etal-99}   applied the       quasi-steady-state approximation
to reduce the three-dimensional system (\ref{3D-model-2}) into  a two-dimensional approximation system
\begin{subequations}
\label{2D-model-1}
\begin{align}
\frac{d M}{d t} &= \frac{4\nu_m P_{c}^{2}}{4P_{c}^{2}+(P-h(P))^{2}}-k_m M,
\label{2D-model-1-1}
\\
\frac{d P}{d t} &= \nu_p M-\frac{(k_1-k_2)h(P)+k_2P}{J_p+P}-k_3P,
\label{2D-model-1-2}
\end{align}
\end{subequations}
where the function $h$ is given by $h(P)=(\sqrt{1+8KP}-1)/(4K)$ for $P\geq 0$
and the constant $K=k_{a}/k_{d}$.
However,
it is explained in \cite{Boieetal-16} that the reduced model obtained by the quasi-steady-state approximation
may possess different dynamical behaviours compared to the full system.
The differences between the original system (\ref{3D-model-2}) and the approximation system (\ref{2D-model-1})
were also numerically compared in \cite{Goussis-Najm-06}.

Here our objective is to  investigate the problem of lower dimensional reduction for system (\ref{3D-model-2})
by the {\it Geometric Singular Perturbation Theory} (abbreviated as GSPT).
Under the assumption that
the rates $k_{a}$ and $k_{d}$ are larger than other rate parameters,
we can rewrite system (\ref{3D-model-2}) as a  slow-fast system.
The key step is to deal with the problem that the critical manifold is noncompact.
The remedy is to construct a compact global attractor in the set $\mathcal{A}_{+}$,
which is obtained from an iterative argument.
We further make use of GSPT in this compact attractor and obtain a two-dimensional slow manifold,
the locally stable manifold of which fills the whole compact region,
and is foliated invariantly with base points in the persistent slow manifold.
As a result, the dynamics of three-dimensional system (\ref{3D-model-2}) is governed
by its restriction to the slow manifold.
In addition, the reduced system  is a regular perturbation of the approximation system (\ref{2D-model-1}).
Here our results lead to a significantly different presentation from those in \cite{Tyson-etal-99}.
Our reduction based on GSPT permits the exact form for the reduced system,
instead of an approximation in the form (\ref{2D-model-1}).
On the other hand,
a compact submanifold of the persistent normally hyperbolic slow manifold is
a global attractor in the region with biological sense.
More specifically, the  main result is stated in the following.
\vskip 0.3cm
\renewcommand\thetheorem{A}
\begin{theorem}
\label{thm-2D-reduction}
{\rm (Two-dimensional reduction).}
Let $\nu_m=\varepsilon k_d \widetilde{\nu}_m$, $k_m=\varepsilon k_d \widetilde{k}_m$,
$\nu_p=\varepsilon k_d \widetilde{\nu}_p$, $k_{i}=\varepsilon k_d \widetilde{k}_{i}$, $i=1,2,3$, and $k_{a}=Kk_{d}$.
Assume that  the parameter $\varepsilon$  satisfies
$0<\varepsilon \leq \varepsilon_{0}$ for a sufficiently small $\varepsilon_{0}>0$.
Then the dynamics of the three-dimensional circadian oscillator model (\ref{3D-model-2}) in the set
$\mathcal{A}_{+}=\left\{(M,P,P_1)\in \mathbb{R}^{3}: M\geq 0, P\geq P_1\geq 0\right\}$
is governed by the following two-dimensional system
\begin{subequations}
\label{2D-model-3}
\begin{align}
\frac{d M}{d t} &=
    \frac{4\nu_m P_{c}^{2}}{4P_{c}^{2}+(P-h(P))^{2}}-k_m M
    +\varepsilon\frac{8\nu_m P_{c}^{2}q_{1}(M,P)(P-h(P))}{(4P_{c}^{2}+(P-h(P))^{2})^{2}}+O(\varepsilon^{2}),
\label{2D-model-3-1}
\\
\frac{d P}{d t} &=
     \nu_p M-\frac{(k_1-k_2)h(P)+k_2P}{J_p+P}-k_3P
    +\varepsilon \frac{(k_2-k_1)q_{1}(M,P)}{J_p+P}+O(\varepsilon^{2}),
\label{2D-model-3-2}
\end{align}
\end{subequations}
where the smooth functions $h$ and $q_{1}$ are defined by
\begin{eqnarray}
h(P)\!\!\!&=&\!\!\!\frac{\sqrt{1+8KP}-1}{4K},\nonumber\\
q_{1}(M,P)\!\!\!&=&\!\!\!\frac{8K\widetilde{\nu}_p MP}{(1+8KP)(1+\sqrt{1+8KP})}
           +\frac{8KP^{2}(\widetilde{k}_2-2\widetilde{k}_1-J_p\widetilde{k}_3-\widetilde{k}_3P)}{(J_p+P)(1+8KP)(1+\sqrt{1+8KP})^{2}},
           \label{df-q-1}
\end{eqnarray}
for $M\geq 0$ and $P\geq 0$.
\end{theorem}
Theorem \ref{thm-2D-reduction} shows that
the effective dynamics of system (\ref{3D-model-2}) with sufficiently small $\varepsilon$  can be obtained
by investigating the reduced system (\ref{2D-model-3}).
Note that the reduced system (\ref{2D-model-3}) with sufficiently small $\varepsilon$ is
a regular perturbation of the approximation system (\ref{2D-model-1}),
then we can  obtain the partial dynamics of three-dimensional system (\ref{3D-model-2}) in the view of regular perturbation theory
(see, for instance, \cite{Chow-Hale-82}).

This paper is organized as follows.
In section \ref{sec-GSPT},
we recall several fundamental results on GSPT as preparations.
In section \ref{sec-reduct},
we construct a compact attractor and apply GSPT to realize the two-dimensional reduction of system (\ref{3D-model-2}).
We give some remarks in the final section.

\section{Geometric singular perturbation theory}
\label{sec-GSPT}

In this section, we introduce some basics of GSPT as preparations.
GSPT was laid by Fenichel \cite{Fenichel-79} in 1979,
which was based on the theory of normally hyperbolic invariant manifolds obtained
by Fenichel's series works \cite{Fenichel-71,Fenichel-74,Fenichel-77}.
Since then, GSPT acting as an important geometric method to deal with the singular perturbation problems
has been extended to a wider research field including
blow-up technique \cite{Dumortieretal-Roussarie-96,Krupa-Szmolyan-01SIMA,Krupa-Szmolyan-01JDE},
exchange lemma \cite{Deng-89,Jones-96,Liu-00,Schecter1-08,Schecter2-08} and so on.
GSPT also provides an approach to explain many phenomena in actual applications,
for example,  in population dynamics \cite{Deng-03,Li-Zhu-13},
cellular physiology \cite{Keener-Sneyd-98,Kosiuk-Szmolyan-16,Rubin-Terman}
and  fluid mechanics \cite{DuLiLi-18}.
For more information on GSPT,
we refer to good references \cite{Hek-2010,Jones-95,Kuehn-15,Wiggins-94}.

Singular perturbation problems usually admit a clear separation in different time scales.
A standard slow-fast system with two time scales is in the following form
\begin{subequations}
\label{slow-1}
\begin{align}
\varepsilon\frac{d x}{d \tau} &= \varepsilon \dot x  = f(x,y,\varepsilon),
\label{slow-1-1}
\\
\frac{d y}{d \tau} &= \dot y  = g(x,y,\varepsilon),
\label{slow-1-2}
\end{align}
\end{subequations}
where $(x,y)\in \mathbb{R}^{n}\times\mathbb{R}^{m}$, $\varepsilon$ is a real parameter with $0<|\varepsilon|\ll 1$,
and the functions $f: \mathbb{R}^{n+m+1}\to \mathbb{R}^{n}$ and
$g: \mathbb{R}^{n+m+1}\to \mathbb{R}^{m}$ are sufficiently smooth on $\mathbb{R}^{n+m+1}$.
System (\ref{slow-1}) can be reformulated with a change of time scale $\tau=\varepsilon t$ as
\begin{subequations}
\label{fast-1}
\begin{align}
\frac{d x}{d t} &= x'  = f(x,y,\varepsilon),
\label{fast-1-1}
\\
\frac{d y}{d t}&= y'  = \varepsilon g(x,y,\varepsilon).
\label{fast-1-2}
\end{align}
\end{subequations}
The time scale $\tau$ is said to be slow whereas that for $t$ is fast.
Then we refer to system (\ref{slow-1}) as  {\it the slow system} and to system (\ref{fast-1}) as {\it the fast system}.
For $\varepsilon\neq 0$, systems (\ref{slow-1}) and (\ref{fast-1}) are equivalent.
To obtain the dynamics of system (\ref{slow-1}) or (\ref{fast-1}) for sufficiently small $|\varepsilon|$,
we  consider the limits for both systems as $\varepsilon \to 0$.
Letting $\varepsilon\to 0$ in system (\ref{slow-1}) yields a differential-algebraic system
\begin{subequations}
\label{reduce-1}
\begin{align}
0&=  f(x,y,0),
\label{reduce-1-1}
\\
\dot y  & = g(x,y,0),
\label{reduce-1-2}
\end{align}
\end{subequations}
which is called {\it the reduced system},
and letting $\varepsilon \to 0$ in system (\ref{fast-1}) yields a system given by
\begin{subequations}
\label{layer-1}
\begin{align}
x'  &= f(x,y,0),
\label{layer-1-1}
\\
y'  &= 0,
\label{layer-1-2}
\end{align}
\end{subequations}
which is called {\it the layer problem}.
In system (\ref{layer-1}) the variable $x$ evolves while the variable $y$ can be viewed as a parameter.
Then we refer to $x$ as {\it the fast variable} and  to $y$ as {\it the slow variable}.
The phase state of system (\ref{reduce-1}) is defined on the set of singular points of system (\ref{layer-1}),
that is, the set $\mathcal{C}_{0}:=\{(x,y)\in \mathbb{R}^{n}\times\mathbb{R}^{m}: f(x,y,0)=0\}$.
This set is called  {\it the critical set}.
Additionally, if it is a submanifold of $\mathbb{R}^{n}\times\mathbb{R}^{m}$,
then it is called the {\it critical manifold}.

A compact submanifold $\mathcal{M}_0$ (with or without boundary) of the set $\mathcal{C}_{0}$
is said to be {\it normally hyperbolic} \cite{Fenichel-79}
if the Jacobian matrix $\partial f/ \partial x$ at each point in $\mathcal{M}_0$ has no eigenvalues on the imaginary axis.
Theorem 9.1 in \cite{Fenichel-79} implies that for a compact normally hyperbolic invariant manifold $\mathcal{M}_0$,
there exists a {\it slow manifold} $\mathcal{M}_{\varepsilon}$ of system (\ref{fast-1}),
which lies within $O(\varepsilon)$ of $\mathcal{M}_{0}$  in the $C^{1}$ topology.
Moreover, the slow manifold $\mathcal{M}_{\varepsilon}$ is {\it locally invariant} under the flow (\ref{fast-1}).
A set $U$ is called to be {\it locally invariant} under the flow (\ref{fast-1})
if it has a neighborhood $V$ so that no trajectories can leave the set $U$ without also leaving $V$.
The slow manifold $\mathcal{M}_{\varepsilon}$ possesses
a locally stable manifold $\mathcal{W}^{s}_{\varepsilon}(\mathcal{M}_{\varepsilon})$
and a locally unstable manifold $\mathcal{W}^{u}_{\varepsilon}(\mathcal{M}_{\varepsilon})$,
which respectively lie within $O(\varepsilon)$, in the $C^{1}$ topology, of the locally stable manifold $\mathcal{W}^{s}_{0}(\mathcal{M}_{0})$
and the locally  unstable manifold $\mathcal{W}^{u}_{0}(\mathcal{M}_{0})$ of the critical manifold $\mathcal{M}_{0}$.
The locally stable manifold $\mathcal{W}^{s}_{\varepsilon}(\mathcal{M}_{\varepsilon})$
and the locally unstable manifold $\mathcal{W}^{u}_{\varepsilon}(\mathcal{M}_{\varepsilon})$
are respectively foliated by the fibers with base points in the slow manifold $\mathcal{M}_{\varepsilon}$.

\section{Reduction of the 3D circadian oscillator model}
\label{sec-reduct}

In this section,
we first establish the existence of a compact attractor in the region with biological meaning.
Then  under the assumption that
the parameters $k_a$ and $k_d$ are sufficiently large compared to other rate constants in system (\ref{3D-model-2}),
we reduce three-dimensional system (\ref{3D-model-2}) to a two-dimensional system via GSPT.

\subsection{Positive invariant sets and attraction}

As preliminaries, we first consider the positive invariant sets of systems (\ref{3D-model-1}) and (\ref{3D-model-2}).
Further, the detailed study on the attraction of a positive invariant set for system (\ref{3D-model-2}) is given.
The similar result for system (\ref{3D-model-1}) can be obtained by taking the transformation $M=M$, $P_1=P_1$ and $P=P_1+2P_2$.
Let $\{\widetilde{\Phi}_{t}: t\in \mathbb{R}\}$ (resp. $\{\Psi_{t}: t\in \mathbb{R}\}$) denote the flow of
system (\ref{3D-model-1}) (resp. (\ref{3D-model-2})).
A set  $\Lambda\subset \mathbb{R}^{3}$ is called the positive invariant set of system (\ref{3D-model-1}) (resp. (\ref{3D-model-2}))
if $\widetilde{\Phi}_{t}\Lambda \subset \Lambda $ (resp. $\Psi_{t} \Lambda \subset \Lambda$) for all $t\geq 0$.
The statement on the positive invariant set of system (\ref{3D-model-1}) is stated in the next lemma.
\begin{lemma}\label{lm-p-inv-1}
The cone $\mathbb{R}^{3}_{+}:=\{(M,P_1,P_2)\in \mathbb{R}^{3}: M\geq 0, P_1\geq 0, P_2 \geq 0\}$
is a positive invariant set of system (\ref{3D-model-1}).
\end{lemma}
\begin{proof}
We first choose three increasing sequences
$\{M^{(n)}\}_{n=0}^{+\infty}$, $\{P_{1}^{(n)}\}_{n=0}^{+\infty}$ and $\{P_{2}^{(n)}\}_{n=0}^{+\infty}$ in the following way.
Take a sequence $\{P_{1}^{(n)}\}_{n=0}^{+\infty}$
which increases to zero and satisfies $P_{1}^{(0)}<P_{1}^{(n)}<P_{1}^{(n+1)}<0$ for each $n\geq 1$
and $P_{1}^{(0)}\in (\max\{-k_3/(2k_a), -J_{p}\},0)$.
Then
$$k_3P_{1}^{(n)}+2k_a (P_{1}^{(n)})^{2}<0, \ \ \ n\geq 0.$$
By continuity we can take the increasing $M^{(n)}$ and $P_2^{(n)}$
with $M^{(n)}<0$ and $P_2^{(n)}<0$ such that
\begin{eqnarray*}
\nu_p M^{(n)}-k_3P_1^{(n)}-2k_a (P_{1}^{(n)})^{2}+2k_d P_2^{(n)}>0, \ \ \
J_p+P_1^{(n)}+2P_2^{(n)}>0.
\end{eqnarray*}
Since $P_{1}^{(n)}<0$ and $P_{1}^{(n)}\to 0$ as $n\to +\infty$,
then  $M^{(n)}\to 0$ and $P_{2}^{(n)}\to 0$ as $n\to +\infty$.
Define a family of cones $\mathcal{C}^{(n)}_{+}$ by
$$\mathcal{C}^{(n)}_{+}:=\{(M,P_1,P_2)\in \mathbb{R}^{3}: M\geq M^{(n)}, P_1\geq P_{1}^{(n)}, P_2 \geq P_2^{(n)}\}.$$
We claim that the cones $\mathcal{C}^{(n)}_{+}$ are the positive invariant sets of system (\ref{3D-model-1}).
In fact, the field vector of system (\ref{3D-model-1})
on the planes $M=M^{(n)}$, $P_1=P_{1}^{(n)}$ and $P_2=P_2^{(n)}$ satisfies the following inequalities:
\begin{eqnarray*}
\frac{d M}{d t}|_{M=M^{(n)}}
        \!\!&=&\!\!  \frac{\nu_m}{1+(P_2/P_c)^{2}}-k_m M^{(n)}>0,\\
\frac{d P_1}{d t}|_{{M\geq M^{(n)}},P_1=P_{1}^{(n)}}
        \!\!&>&\!\!   \nu_p M^{(n)}-k_3P_1^{(n)}-2k_a (P_{1}^{(n)})^{2}+2k_d P_2^{(n)}>0,\\
\frac{d P_2}{d t}|_{P_1\geq P_{1}^{(n)},P_2=P_2^{(n)}}
        \!\!&=&\!\! k_a (P_{1}^{(n)})^{2}-k_d P_2^{(n)}-\frac{k_2P_2^{(n)}}{J_p+P_1^{(n)}+2P_2^{(n)}}-k_3P_2^{(n)}>0,
\end{eqnarray*}
then there are no solutions of system (\ref{3D-model-1}) starting from $\mathcal{C}^{(n)}_{+}$
and leaving them by crossing their boundaries.
Thus the claim is true.
Recall that $M^{(n)}$, $P_{1}^{(n)}$ and $P_2^{(n)}$ are monotonically increasing to zero,
then  the cones $\mathcal{C}^{(n)}_{+}$ approach to $\mathbb{R}^{3}_{+}$ as $n\to+\infty$.
Thus by the positive invariant property of the cones $\mathcal{C}^{(n)}_{+}$,
the cone  $\mathbb{R}^{3}_{+}$ is also positive invariant.
Therefore, the proof is now complete.
\end{proof}

Note that  $P=P_1+2 P_2$ in system (\ref{3D-model-2}),
then as a direct corollary of Lemma \ref{lm-p-inv-1}
we have the following result on the positive invariant set of system (\ref{3D-model-2}).

\begin{corollary}\label{Cor-p-inv-2}
The cone $\mathcal{A}_{+}=\left\{(M,P,P_1)\in \mathbb{R}^{3}: M\geq 0, P\geq P_1\geq 0\right\}$
is a positive invariant set of system (\ref{3D-model-2}).
\end{corollary}

We remark that all initial values with biological meaning are in the cone $\mathcal{A}_{+}$.
Thus throughout this paper, we only consider the dynamics of system (\ref{3D-model-2}) in the cone $\mathcal{A}_{+}$.
Let $d:\mathbb{R}^{3}\times \mathbb{R}^{3}\to [0,+\infty)$ denote a metric on $\mathbb{R}^{3}$.
We say that a set $A\subset \mathbb{R}^{3}$ attracts a set $B\subset \mathbb{R}^{3}$ under the flow $\{\Psi_{t}: t\in \mathbb{R}\}$,
if $\lim_{t\to+\infty} dist(\Psi_{t}B, A)=0$,
where $dist(\Psi_{t}B, A)=\sup_{x\in \Psi_{t}B} \inf_{y\in A} d(x,y)$.
We next establish the existence of a smaller attracting set inside $\mathcal{A}_{+}$.

\begin{lemma}\label{lm-p-inv-2}
Let the set $\mathcal{A}^{1}_{+}$ be given by
\begin{eqnarray*}
\mathcal{A}^{1}_{+}\!\!\!&=&\!\!\!
  \left\{(M,P,P_1)\in \mathbb{R}^{3}:
        0\leq M\leq \frac{\nu_{m}}{k_{m}}, \ \
        0\leq P_1 \leq P \leq \frac{\nu_{m}\nu_{p}}{k_3k_{m}}\right\}.
\end{eqnarray*}
Then $\mathcal{A}^{1}_{+}$ attracts $\mathcal{A}_{+}$ under the flow $\{\Psi_{t}: t\in \mathbb{R}\}$,
and the set $\mathcal{A}^{1}_{+}$ is a positive invariant set of system (\ref{3D-model-2}).
\end{lemma}
\begin{proof}
For any initial value $(M(0),P(0),P_1(0))\in \mathcal{A}_{+}$,
let $(M(t),P(t),P_1(t))$ denote the solution of system (\ref{3D-model-2}) with this initial value.
Then by the Corollary \ref{Cor-p-inv-2} we have $M(t)\geq 0$, $P(t)\geq 0$ and $P_1(t)\geq 0$ for all $t\geq 0$.
From (\ref{3D-model-2-1}) it follows that $$\frac{d M}{d t} \leq \nu_m -k_m M.$$
Then by {\it Gronwall's Inequality} we obtain
\begin{eqnarray}\label{est-Mt}
M(t)\leq M(0)e^{-k_m t}+\nu_{m}\int_{0}^{t}e^{k_m(s-t)} ds
    <M(0)e^{-k_m t}+\frac{\nu_{m}}{k_{m}}, \ \ \ t\geq 0.
\end{eqnarray}
Since $P\geq P_{1}$, $\nu_{p}>0$ and $k_{i}>0$ in (\ref{3D-model-2-2}),
then $$\frac{d P}{d t} \leq \nu_p M-k_3P,$$
 which implies
\begin{eqnarray}\label{est-Pt-1}
P(t)\leq P(0)e^{-k_3 t}+\nu_{p}\int_{0}^{t} M(s) e^{k_{3}(s-t)} ds, \ \ \ t\geq 0.
\end{eqnarray}
According to $k_3\neq k_{m}$ and $k_3= k_{m}$,
we give two different estimates for $P(t)$ as follow.
In case $k_3\neq k_{m}$, substituting (\ref{est-Mt}) into (\ref{est-Pt-1}) yields that
\begin{eqnarray}\label{est-Pt-2-1}
P(t) < P(0)e^{-k_3 t}+\frac{\nu_{p}M(0)}{k_3-k_m}\left(e^{-k_mt}-e^{-k_3t}\right)+\frac{\nu_{m}\nu_{p}}{k_3k_m}, \ \ \ t\geq 0.
\end{eqnarray}
Similarly, in case $k_3= k_{m}$,
\begin{eqnarray}\label{est-Pt-2-2}
P(t) < P(0)e^{-k_3 t}+\nu_{p}M(0)te^{-k_3t}+\frac{\nu_{p}\nu_{m}}{(k_3)^{2}}, \ \ \ t\geq 0.
\end{eqnarray}
By (\ref{3D-model-2-3}) we obtain
\begin{eqnarray}\label{est-P1t-1}
P_1(t)\leq P_1(0)e^{-(k_3+k_d)t}+\int_{0}^{t}(\nu_{p}M(s)+k_dP(s))e^{(k_3+k_d)(s-t)} ds, \ \ \ t\geq 0.
\end{eqnarray}
Then substituting (\ref{est-Mt}), (\ref{est-Pt-2-1}) or (\ref{est-Pt-2-2}) into (\ref{est-P1t-1})
yields that for $t\geq 0$, the following estimates hold.
For $k_m>k_{3}$ and $k_3+k_{d}\neq k_{m}$,
\begin{eqnarray*}
P_1(t)\!\!\!&<&\!\!\!P_1(0)e^{-(k_3+k_d)t}+\frac{\nu_{p}M(0)}{k_3+k_{d}-k_m}\left(e^{-k_mt}-e^{-(k_3+k_{d})t}\right)\\
     \!\!\!&&\!\!\!+\left(P(0)+\frac{\nu_{p}M(0)}{k_m-k_3}\right)e^{-k_3 t}+\frac{\nu_{m}\nu_{p}}{k_3k_{m}}.
\end{eqnarray*}
For $k_m>k_{3}$ and $k_3+k_{d}=k_{m}$,
\begin{eqnarray*}
P_1(t)<P_1(0)e^{-k_{m}t}+\nu_{p}M(0)te^{-k_{m}t}
+\left(P(0)+\frac{\nu_{p}M(0)}{k_d}\right)e^{-k_3 t}+\frac{\nu_{m}\nu_{p}}{k_3k_{m}}.
\end{eqnarray*}
For $k_3>k_{m}$,
\begin{eqnarray*}
P_1(t)<P_1(0)e^{-(k_3+k_d)t}+\frac{\nu_{p}M(0)}{k_3-k_m}e^{-k_mt}
+P(0)e^{-k_3 t}+\frac{\nu_{m}\nu_{p}}{k_3k_{m}}.
\end{eqnarray*}
For $k_3= k_{m}$,
\begin{eqnarray*}
P_1(t)<P_1(0)e^{-(k_3+k_d)t}+\left(P(0)+\frac{\nu_{p}M(0)}{k_d}(k_dt+e^{-k_{d}t}+1)\right)e^{-k_3 t}+\frac{\nu_{m}\nu_{p}}{(k_3)^{2}}.
\end{eqnarray*}
Hence we have
$$\lim_{t\to +\infty}dist((M(t),P(t),P_1(t)), \mathcal{A}_{+}^{1})=0.$$
Thus, the set $\mathcal{A}^{1}_{+}$ attracts the set $\mathcal{A}_{+}$ under the flow $\{\Psi_{t}: t\in \mathbb{R}\}$.

Note that for $M>\nu_{m}/k_{m}$ we have
$$
\frac{d M}{d t}
        <\frac{4\nu_m P_{c}^{2}}{4P_{c}^{2}+(P-P_1)^{2}}-\nu_m \leq 0,
$$
then the set $\mathcal{A}^{2}_{+}:=\{(M, P, P_1)\in \mathbb{R}^{3}: 0\leq M\leq \nu_{m}/k_{m}\}\cap \mathcal{A}_{+}$
is a positive invariant set of system (\ref{3D-model-2}).
Let the field vector of system (\ref{3D-model-2}) be restricted to the set $\mathcal{A}^{2}_{+}\cap \{P=\nu_{m}\nu_{p}/(k_3k_{m})\}$.
Then on this set we have
$$
\frac{(k_1-k_2)P_1+k_2P}{J_p+P}>0 \ \mbox{ and }\
\frac{d P}{d t}<\nu_{m}M-k_3P\leq  \frac{\nu_{m}\nu_p}{k_{m}}-\frac{\nu_{m}\nu_{p}}{k_{m}}=0.
$$
By the positive invariant property of the set $\mathcal{A}^{2}_{+}$,
we get that  the set $\mathcal{A}^{1}_{+}$ is also positive invariant.
Therefore, the proof is now complete.
\end{proof}

\subsection{A two-dimensional reduction}

In this section,
we reduce the three-dimensional circadian oscillator model (\ref{3D-model-2})
to a two-dimensional system on the persistent normally hyperbolic slow manifold,
under the assumption that the parameters $k_{a}$ and $k_{d}$ are large compared to other parameters.

Let  $\nu_m=\varepsilon k_d \widetilde{\nu}_m$, $k_m=\varepsilon k_d \widetilde{k}_m$,
$\nu_p=\varepsilon k_d \widetilde{\nu}_p$, $k_{i}=\varepsilon k_d \widetilde{k}_{i}$, $i=1,2,3$, and $k_{a}=Kk_{d}$,
where the parameter $\varepsilon$ satisfies $0<\varepsilon \leq \varepsilon_{0}\ll 1$ for some $\varepsilon_{0}>0$
from the assumption that the parameters $k_a$ and $k_d$ are sufficiently large compared to other rate constants.
By a time rescaling $t=\widetilde{\tau}/k_{d}$, system (\ref{3D-model-2}) becomes
\begin{subequations}
\label{3D-model-3}
\begin{align}
\frac{d M}{d \widetilde{\tau}} &= \varepsilon\left(\frac{4\widetilde{\nu}_m P_{c}^{2}}{4P_{c}^{2}+(P-P_1)^{2}}-\widetilde{k}_m M\right),
\label{3D-model-3-1}
\\
\frac{d P}{d \widetilde{\tau}} &= \varepsilon\left( \widetilde{\nu}_p M-\frac{(\widetilde{k}_1-\widetilde{k}_2)P_1+\widetilde{k}_2P}{J_p+P}-\widetilde{k}_3P\right),
\label{3D-model-3-2}
\\
\frac{d P_1}{d \widetilde{\tau}} &= \varepsilon\left( \widetilde{\nu}_p M-\frac{\widetilde{k}_1P_1}{J_p+P}-\widetilde{k}_3P_1\right)-2K P_1^{2}- P_1+P.
\label{3D-model-3-3}
\end{align}
\end{subequations}
Let $\{\widetilde{\Psi}_{\widetilde{\tau}}: \widetilde{\tau}\in \mathbb{R}\}$ denote the dynamical system generated by system (\ref{3D-model-3}).
The layer problem is obtained by setting $\varepsilon=0$ in system (\ref{3D-model-3}), that is,
\begin{subequations}
\label{3D-model-layer}
\begin{align}
\frac{d M}{d \widetilde{\tau}} &= 0,
\label{3D-model-layer-1}
\\
\frac{d P}{d \widetilde{\tau}} &= 0,
\label{3D-model-layer-2}
\\
\frac{d P_1}{d \widetilde{\tau}} &= -2K P_1^{2}- P_1+P.
\label{3D-model-layer-3}
\end{align}
\end{subequations}
Recall that $\mathcal{A}^{1}_{+}$ attracts $\mathcal{A}_{+}$ under the flow $\{\Psi_{t}: t\in \mathbb{R}\}$,
which is stated in Lemma \ref{lm-p-inv-2}.
Then we take a critical manifold $\mathcal{M}_{0}$ in the form
\begin{eqnarray*}
\ \mathcal{M}_{0}:=\left\{(M, P, P_1)\in \mathbb{R}^{3}:P_1=h(P),\
                 -\varrho_{1}\leq M\leq \frac{\nu_{m}}{k_{m}}+\varrho_{1},\
                 -\varrho_{1}\leq P \leq \frac{\nu_{m}\nu_{p}}{k_3k_{m}}+\varrho_{1} \right\},
\end{eqnarray*}
where the constant $\varrho_{1}$ is in the interval $(0,1/(8K))$ and  the function $h$ is in the form
\begin{eqnarray*}
h(P)=\frac{\sqrt{1+8KP}-1}{4K}, \ \  \ P\geq -\frac{1}{8K}.
\end{eqnarray*}
For each point $(M, P, P_1)$ in the set $\mathcal{M}_{0}$,
we observe that the linearized system at this point has two zero eigenvalues
and one nonzero eigenvalue $\mu=-\sqrt{1+8KP}<0$ for $-\varrho_{1}\leq P \leq \nu_{m}\nu_{p}/(k_3k_{m})+\varrho_{1}$,
whose eigenvector is $\xi=(0,0,1)$.
Then the critical manifold $\mathcal{M}_{0}$ is normally hyperbolic.
Let $\mathcal{W}^{s}_{0}(\mathcal{M}_{0})$ denote the stable manifold of $\mathcal{M}_{0}$ for system (\ref{3D-model-layer}).
Then the stable manifold $\mathcal{W}^{s}_{0}(\mathcal{M}_{0})$ is three-dimensional,
and $\mathcal{W}^{s}_{0}(\mathcal{M}_{0})$ transversally intersects with the plane $P_1=P_1^{*}$ for each $P_1^{*}\in \mathbb{R}$
at points with
\begin{eqnarray*}
-\varrho_{1}\leq M\leq \frac{\nu_{m}}{k_{m}}+\varrho_{1},\ \ \
-\varrho_{1}\leq P \leq \frac{\nu_{m}\nu_{p}}{k_3k_{m}}+\varrho_{1} \ \ \
\mbox{ and }\  P\neq 2K (P_1^{*})^{2}+P_1^{*}.
\end{eqnarray*}

By applying the theory of normally hyperbolic invariant manifolds
\cite{Fenichel-71,Fenichel-74,Fenichel-77,Fenichel-79,Jones-95},
we can obtain the following results on the local dynamics of system (\ref{3D-model-3})
with sufficiently small $\varepsilon$ near the critical manifold $\mathcal{M}_{0}$.
\begin{lemma}\label{lm-Fenichel}
Assume that  the parameter $\varepsilon$  in system (\ref{3D-model-3}) satisfies
$0<\varepsilon \leq \varepsilon_{0}$ for a sufficiently small $\varepsilon_{0}>0$.
Then for each $ r$ in $ (0, +\infty)$, the following statements hold:
\item{\bf (i)}
There exists a locally invariant manifold  $\mathcal{M}_{\varepsilon}$
which is $O(\varepsilon)$-close and diffeomorphic to the critical manifold $\mathcal{M}_{0}$.
Furthermore, there is a $C^{r}$ smooth function
$$q(M,P,\varepsilon)=h(P)+\varepsilon q_{1}(M,P)+O(\varepsilon^{2})$$
such that the manifold $\mathcal{M}_{\varepsilon}$ is given by
\begin{eqnarray*}
\mathcal{M}_{\varepsilon}:=\!\left\{\!(M, P, P_1)\in \mathbb{R}^{3}\!:P_1=q(M,P,\varepsilon),\,
                 -\varrho_{1}\leq M\leq \frac{\nu_{m}}{k_{m}}+\varrho_{1},\,
                 -\varrho_{1}\leq P \leq \frac{\nu_{m}\nu_{p}}{k_3k_{m}}+\varrho_{1}\! \right\},
\end{eqnarray*}
where $h(P)=(\sqrt{1+8KP}-1)/(4K)$ and  $q_{1}$ is given by (\ref{df-q-1}).

\item{\bf (ii)}
There exists, in some small neighborhood $\mathcal{U}$ of $\mathcal{M}_{0}$
which is independent of $\varepsilon$ with $0<\varepsilon\leq \varepsilon_{0} $,
a locally invariant stable manifold $\mathcal{W}^{s}_{\varepsilon}(\mathcal{M}_{\varepsilon})$ of the manifold $\mathcal{M}_{\varepsilon}$ which is $O(\varepsilon)$-close and diffeomorphic to $\mathcal{W}^{s}_{0}(\mathcal{M}_{0})$.
Furthermore, the locally invariant manifold $\mathcal{W}^{s}_{\varepsilon}(\mathcal{M}_{\varepsilon})$ is $C^{r}$ smooth in $M, P, P_1$ and $\varepsilon$.

\item{\bf (iii)}
For each $u_{\varepsilon}\in \mathcal{M}_{\varepsilon}$,
there exists a submanifold $\mathcal{W}^{s}_{\varepsilon}(u_{\varepsilon})$ of the manifold $\mathcal{W}^{s}_{\varepsilon}(\mathcal{M}_{\varepsilon})$
which is $O(\varepsilon)$-close and  diffeomorphic to $\mathcal{W}^{s}_{0}(u_{0})$.
Furthermore, the family $\{\mathcal{W}^{s}_{\varepsilon}(u_{\varepsilon}): u_{\varepsilon}\in \mathcal{M}_{\varepsilon}\}$
is invariant in the sense that
the manifold $\mathcal{W}^{s}_{\varepsilon}(u_{\varepsilon})$ is mapped by the time $\widetilde{\tau}$ flow to another submanifold
$\mathcal{W}^{s}_{\varepsilon}(u_{\varepsilon}\cdot \widetilde{\tau})$
whose base point is the image of the time $\widetilde{\tau}$ flow for $u_{\varepsilon}$.
In addition, if $u^{s}_{\varepsilon}\in \mathcal{W}^{s}_{\varepsilon}(\mathcal{M}_{\varepsilon})$,
then there is a point $u_{\varepsilon}\in \mathcal{M}_{\varepsilon}$
such that
$|\widetilde{\Psi}_{\widetilde{\tau}}u^{s}_{\varepsilon}-\widetilde{\Psi}_{\widetilde{\tau}}u_{\varepsilon}|
\leq \kappa_{s}e^{-\mu_{s}\widetilde{\tau}}$ for $\widetilde{\tau}\geq 0$
and some positive constants $\kappa_{s}$ and $\mu_{s}$.
\end{lemma}
\begin{proof}
We only give the proof of the representation for the manifold $\mathcal{M}_{\varepsilon}$,
that is, the expression of the smooth function $q$ given in {\bf (i)},
the remaning statements are  due to\cite[Theorem 9.1]{Fenichel-79}.
By the invariance of the manifold $\mathcal{M}_{\varepsilon}$,
we have
\begin{eqnarray*}
\frac{d P_1}{d \widetilde{\tau}}=
     \frac{\partial q(M,P,\varepsilon)}{\partial M}\frac{d M}{d \widetilde{\tau}}
     +\frac{\partial q(M,P,\varepsilon)}{\partial P}\frac{d P}{d \widetilde{\tau}}.
\end{eqnarray*}
By substituting (\ref{3D-model-3-1}), (\ref{3D-model-3-2}) and $P_1=q(M,P,\varepsilon)$ into the above equality
and comparing coefficient of powers of $\varepsilon$,
we obtain
\begin{eqnarray*}
q(M,P,0)=\frac{\sqrt{1+8KP}-1}{4K}, \ \ \ \frac{\partial q}{\partial \varepsilon}(M,P,0)=q_{1}(M,P).
\end{eqnarray*}
Therefore, the proof is now complete.
\end{proof}

Following Lemma \ref{lm-Fenichel},
we further obtain a smaller attractor in $\mathcal{M}_{\varepsilon}$
and realize the reduction of three-dimensional system (\ref{3D-model-3}).
More precisely, we have the following results.

\begin{lemma}\label{prop-2D-reduc}
Assume that  the parameter $\varepsilon$ satisfies
$0<\varepsilon \leq \varepsilon_{0}$ for a sufficiently small $\varepsilon_{0}>0$.
Let the set $\widetilde{\mathcal{A}}_{\varepsilon}$ be defined by
$\widetilde{\mathcal{A}}_{\varepsilon}=\mathcal{A}^{1}_{+}\cap \mathcal{M}_{\varepsilon}$,
where $\mathcal{A}^{1}_{+}$ and $\mathcal{M}_{\varepsilon}$ are given, respectively,
in Lemmas \ref{lm-p-inv-2} and \ref{lm-Fenichel}.
Then  $\widetilde{\mathcal{A}}_{\varepsilon}$ is a positive invariant set of system (\ref{3D-model-3})
and  attracts the set $\mathcal{A}_{+}$ under the flow $\{\widetilde{\Psi}_{\widetilde{\tau}}: \widetilde{\tau}\in \mathbb{R}\}$.
Furthermore,
system (\ref{3D-model-3}) restricted to $\widetilde{\mathcal{A}}_{\varepsilon}$ is given by
\begin{subequations}
\label{2D-model-2}
\begin{align}
\frac{d M}{d \widetilde{\tau}} &=
    \varepsilon\left(\frac{4\widetilde{\nu}_m P_{c}^{2}}{4P_{c}^{2}+(P-h(P))^{2}}-\widetilde{k}_m M
    +\varepsilon\frac{8\widetilde{\nu}_m P_{c}^{2}q_{1}(M,P)(P-h(P))}{(4P_{c}^{2}+(P-h(P))^{2})^{2}}+O(\varepsilon^{2})\right),
\label{2D-model-2-1}
\\
\frac{d P}{d \widetilde{\tau}} &=
    \varepsilon\left( \widetilde{\nu}_p M-\frac{(\widetilde{k}_1-\widetilde{k}_2)h(P)+\widetilde{k}_2P}{J_p+P}-\widetilde{k}_3P
    +\varepsilon \frac{(\widetilde{k}_2-\widetilde{k}_1)q_{1}(M,P)}{J_p+P}+O(\varepsilon^{2})\right),
\label{2D-model-2-2}
\end{align}
\end{subequations}
where the function $q_1$ is given by (\ref{df-q-1})
and the function $h$ is defined by $h(P)=(\sqrt{1+8KP}-1)/(4K)$ for $0\leq P\leq \nu_{m}\nu_{p}/(k_3k_{m})$.
\end{lemma}
\begin{proof}
By Lemma \ref{lm-Fenichel} we obtain that for a sufficiently small $\varepsilon$,
the critical manifold $\mathcal{M}_{0}$ perturbs smoothly to the slow manifold $\mathcal{M}_{\varepsilon}$
and is  $O(\varepsilon)$-close to  $\mathcal{M}_{0}$ in the $C^{1}$ topology.
Since
\begin{eqnarray*}
\mathcal{A}^{1}_{+}\cap \mathcal{M}_{0}\!\!\!&=&\!\!\!
  \left\{(M,P,P_1)\in \mathbb{R}^{3}: P_1=h(P),\
        0\leq M\leq \frac{\nu_{m}}{k_{m}}, \
        0\leq P \leq \frac{\nu_{m}\nu_{p}}{k_3k_{m}}\right\},
\end{eqnarray*}
then $\widetilde{\mathcal{A}}_{\varepsilon}=\mathcal{A}^{1}_{+}\cap \mathcal{M}_{\varepsilon}$ is not empty
for sufficiently small $\varepsilon$.

For each $(M(0),P(0),P_1(0))\in \mathbb{R}^{3}$,
let $(M(\widetilde{\tau}),P(\widetilde{\tau}),P_1(\widetilde{\tau}))$ be the solution of system (\ref{3D-model-3}) with this initial value.
Assume that $(M(0),P(0),P_1(0))\in \widetilde{\mathcal{A}}_{\varepsilon}=\mathcal{A}^{1}_{+}\cap \mathcal{M}_{\varepsilon}$
and recall that system (\ref{3D-model-2}) is changed into system (\ref{3D-model-3}) by a time rescaling $t=\widetilde{\tau}/k_{d}$,
then by Lemma \ref{lm-p-inv-2} we have $(M(\widetilde{\tau}),P(\widetilde{\tau}),P_1(\widetilde{\tau}))\in \mathcal{A}^{1}_{+}$ for $\widetilde{\tau}\geq 0$,
together with the invariant property of the slow manifold $\mathcal{M}_{\varepsilon}$ from Lemma \ref{lm-Fenichel},
yields that $(M(\widetilde{\tau}),P(\widetilde{\tau}),P_1(\widetilde{\tau}))\in \mathcal{M}_{\varepsilon}\cap \mathcal{A}^{1}_{+}=\widetilde{\mathcal{A}}_{\varepsilon}$ for $\widetilde{\tau}\geq 0$.
Thus, we obtain that the set $\widetilde{\mathcal{A}}_{\varepsilon}$ is a positive invariant set of system (\ref{3D-model-3}).

To prove that the set $\widetilde{\mathcal{A}}_{\varepsilon}$ attracts the set $\mathcal{A}_{+}$
under the flow $\{\widetilde{\Psi}_{\widetilde{\tau}}: \widetilde{\tau}\in \mathbb{R}\}$,
by Lemma \ref{lm-p-inv-2} it suffices to prove
the set $\widetilde{\mathcal{A}}_{\varepsilon}$ attracts the set $\mathcal{A}^{1}_{+}$.
Recall that $\mathcal{A}^{1}_{+}$ is a compact positive invariant set of system (\ref{3D-model-3}),
then by (\ref{3D-model-3-3}) we obtain  that for a sufficiently small $\varepsilon_{0}$,
each orbit starting from $\mathcal{A}^{1}_{+}$ eventually enters the set $\mathcal{U}$ obtained in Lemma \ref{lm-Fenichel} {\bf (ii)}
in a finite time.
It follows from Lemma \ref{lm-Fenichel} {\bf (iii)}  that
along the stable manifold of $\mathcal{M}_{\varepsilon}$,
this orbit decays to the slow manifold $\mathcal{M}_{\varepsilon}$ at an exponential rate for $\widetilde{\tau}\geq 0$.
Consequently, it exponentially decays  to the set $\mathcal{A}^{1}_{+}\cap \mathcal{M}_{\varepsilon}=\widetilde{\mathcal{A}}_{\varepsilon}$.
This proves that the set $\widetilde{\mathcal{A}}_{\varepsilon}$ attracts the set $\mathcal{A}^{1}_{+}$.

By Lemma \ref{lm-Fenichel},
system (\ref{3D-model-3}) restricted to the slow manifold $\mathcal{M}_{\varepsilon}$ is obtained by substituting the function
$q(M,P,\varepsilon)=h(P)+\varepsilon q_{1}(M,P)+O(\varepsilon^{2})$ into (\ref{3D-model-3-1}) and $(\ref{3D-model-3-2})$.
Note that $\widetilde{\mathcal{A}}_{\varepsilon}\subset \mathcal{M}_{\varepsilon}$ is positive invariant under the flow of system (\ref{3D-model-3}),
then  system (\ref{2D-model-2}) can be otained by a direct computation.
Therefore, the proof is now complete.
\end{proof}

Based on the  theory of normally hyperbolic invariant manifolds,
we prove that all solutions with biological meaning of system (\ref{3D-model-3}) with sufficiently small $\varepsilon$
converge to $\widetilde{\mathcal{A}}_{\varepsilon}$ as time goes to infinity. See Figure \ref{Fig-reduced-manifold}.
\begin{figure}[!htbp]
  \centering
  \includegraphics[width=9.5cm]{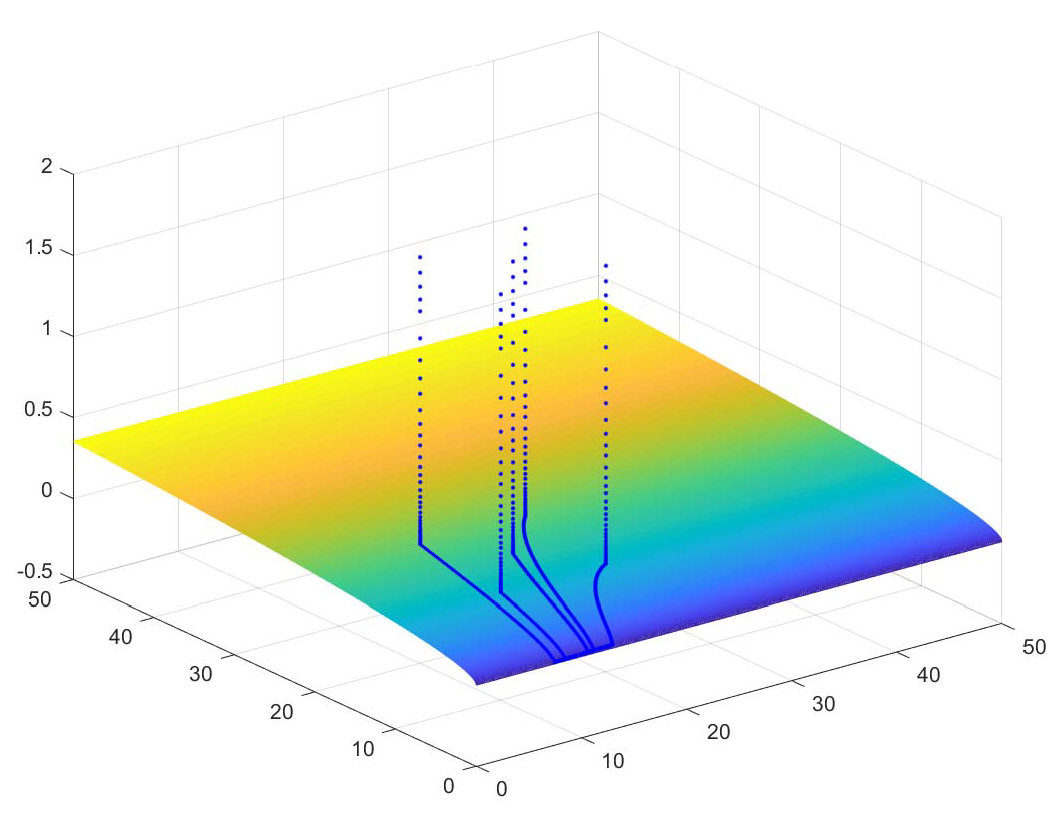}
  \caption{
  The attraction of the slow manifold.
  The surface is the  critical manifolds $\mathcal{M}_{0}$, which is the zeroth-order approximation of the slow manifold,
  and the discrete orbits,  respectively,  start from $(10,10,2)$, $(15,15,2)$, $(20,20,2)$, $(10,20,2)$ and $(20,10,2)$.
  Here $k_{a}=20000$, $k_{d}=100$ and the remaining parameters in (\ref{3D-model-2}) are chosen as in \cite[Table 1, p.2414]{Tyson-etal-99},
  that is, $v_{m}=1$, $k_{3}=k_{m}=0.1$, $v_{p}=0.5$, $k_{1}=10$, $k_{2}=0.03$,  $P_{c}=0.1$ and $J_{p}=0.05$.
  System (\ref{3D-model-3}) with $\widetilde{k}_{2}=1$  has small parameter $\varepsilon=0.0003$.
   }
  \label{Fig-reduced-manifold}
\end{figure}
Hence, the dynamics of the full system (\ref{3D-model-3}) can be obtained
by the studying  the dynamics on the slow manifolds $\mathcal{M}_{\varepsilon}$.
By Lemma \ref{prop-2D-reduc},
we give the main result on the reduction of three-dimensional system (\ref{3D-model-2}).

\noindent{\bf Proof of Theorem \ref{thm-2D-reduction}.}
Note that by taking  a time rescaling $\widetilde{\tau}=k_{d}t$,
system (\ref{3D-model-3}) is changed into system (\ref{3D-model-2}).
Therefore, by Lemma \ref{prop-2D-reduc} the proof for this theorem is finished.
\quad\quad$\Box$

\section{Concluding remarks}

We have considered a three-dimensional circadian rhythm model based on
dimerization and proteolysis of PER and TIM proteins in Drosophila.
Under the assumption that the rates $k_{a}$ and $k_{d}$ are sufficiently large compared to other rate parameters,
we reduce the three-dimensional system to a simpler two-dimensional system via GSPT.
To realize the reduction,
we first establish the existence of a compact attractor.
Then based on {\it Fenichel's Theorem} \cite[Theorem 9.1]{Fenichel-79},
we obtain the reduced system,
which is the restriction of the full system on the persistent normally hyperbolic slow manifold.
The approach in this paper can be also used to simplify many high-dimensional systems
modeling the biological, chemical and physical phenomena.

To analyze the circadian oscillations in the three-dimensional circadian oscillator model,
it is useful to  examine  the dynamics of the reduced system (\ref{2D-model-3})
in the view of the regular perturbation of the zeroth-order approximation.
In particular, it is interesting to investigate the zeroth-order approximation (\ref{2D-model-1}) of the reduced system,
which is equivalent to a Li\'enard-like equation
\begin{subequations}
\label{2D-model-12}
\begin{align}
\frac{d x}{d t} &= y-\left((\delta+1)(x^{2}+2x)+\frac{b_{2}x^{2}+2b_{1}x}{x^{2}+2x+a}\right),
\\
\frac{d y}{d t} &= 2\delta(x+1)\left(\frac{v}{x^{4}+c}-\frac{b_{2}x^{2}+2b_{1}x}{x^{2}+2x+a}-x^{2}-2x\right),
\end{align}
\end{subequations}
where the positive parameters $a, b_{1}, b_{2}, c, \delta, v$ are given by
\begin{eqnarray*}
a=8J_{P}K,\ \
b_{1}=\frac{8k_{2}K}{k_3},\ \
b_{2}=\frac{8k_{1}K}{k_3},\ \
c=256K^{2}P_{c}^{2}, \ \
\delta=\frac{k_{m}}{k_{3}},\ \
v=\frac{2048\nu_{m}\nu_{p}P_{c}^{2}K^{3}}{k_3k_m}.
\end{eqnarray*}
Then the results on the limit cycles of Li\'enard equations (see, for instance, \cite{Dumortieretal06,ZZF-etal})
can be applied to obtain the global dynamics of the Li\'enard-like equation (\ref{2D-model-12}).
It is also important to investigate the difference between the orginal system (\ref{3D-model-1})
and the reduced system (\ref{2D-model-3}).

\end{document}